\documentclass[12pt]{amsart}

\usepackage{amssymb}
\usepackage{latexsym}
\usepackage{verbatim}
\usepackage{fullpage,color}
\usepackage{euscript}

\input {cyracc.def}
\numberwithin{equation}{section}

\usepackage[colorlinks=true,linkcolor=blue,urlcolor=violet,citecolor=magenta]{hyperref}

 \font\tencyr=wncyr10 
\font\tencyi=wncyi10 
\font\tencysc=wncysc10 
\def\rus{\tencyr\cyracc}
\def\rusi{\tencyi\cyracc}
\def\rusc{\tencysc\cyracc}

\newtheorem{thm}{Theorem}[section]

\newtheorem{lm}[thm]{Lemma}
\newtheorem{cl}[thm]{Corollary}
\newtheorem{prop}[thm]{Proposition}

\theoremstyle{remark}
\newtheorem{rmk}[thm]{Remark}

\theoremstyle{definition}
\newtheorem{ex}[thm]{Example}

\newcommand {\g}{{\mathfrak g}}

\newcommand {\n}{{\mathfrak n}}

\newcommand {\p}{{\mathfrak p}}

\newcommand {\es}{{\mathfrak s}}
\newcommand {\te}{{\mathfrak t}}

\newcommand {\slno}{\mathfrak{sl}_{n+1}}

\newcommand {\BC}{{\mathbb C}}

\newcommand {\BZ}{{\mathbb Z}}

\newcommand {\esi}{\varepsilon}
\newcommand{\lb}{\lambda}
\newcommand{\ap}{\alpha}

\renewcommand{\le}{\leqslant}
\renewcommand{\ge}{\geqslant}

\newcommand {\sfr}{\eus R}
\newcommand{\eus}{\EuScript}

\newcommand {\hot}{{\mathsf{ht}}}

\newcommand {\rk}{{\mathsf{rk\,}}}

\newcommand {\GR}[2]{{\textrm{{\bf #1}}}_{#2}}
\newcommand {\GRt}[2]{{\widetilde{\textrm{{\bf #1}}}}_{#2}}

\newcommand {\bbk}{\Bbbk}

\definecolor{my_color}{rgb}{0,0.5,0.5}
\definecolor{MIXT}{rgb}{0.7,0.3,0.3}

\begin{document}
\hfill { {\color{blue}\scriptsize October 16, 2008}}
\vskip1ex

\title[Properties of weight posets]
{Properties of weight posets for weight multiplicity free representations}
\author[D.\,Panyushev]{Dmitri I.~Panyushev}
\address[]{Independent University of Moscow,
Bol'shoi Vlasevskii per. 11, 119002 Moscow, \ Russia
\hfil\break\indent
Institute for Information Transmission Problems, B. Karetnyi per. 19, Moscow 127994
}
\email{panyush@mccme.ru}
\maketitle

\section*{Introduction}

\noindent
Let $\g$ be a semisimple Lie algebra over an algebraically closed field $\bbk$
of characteristic zero, $\te$ a Cartan subalgebra of $\g$, 
 and $\sfr$ a finite-dimensional $\g$-module.
The set of $\te$-weights in $\sfr$ is denoted by  $\eus P(\sfr)$.  
Having chosen a set of simple roots for $(\g,\te)$, we can regard $\eus P(\sfr)$ 
as poset with respect to the {\it root order}. For $\gamma,\mu\in\eus P(\sfr)$, this means 
that $\mu$ {\it covers\/} $\gamma$ if and only if $\mu-\gamma$ is a simple root.
These posets are called {\it  weight posets}.
The {\it Hasse diagram\/} of $\eus P(\sfr)$ is a directed graph whose set of vertices is 
$\eus P(\sfr)$ and there is the edge directed from 
$\gamma$ to $\mu$ if and only if $\mu$ {covers\/} $\gamma$.
The set of edges in the Hasse diagram is denoted by $\eus E(\sfr)$.
We say that  $\sfr$ is  {\it weight multiplicity free\/} (\textsf{wmf} for short) if all $\te$-weight
spaces in $\sfr$ are one-dimensional. Then $\dim\sfr=\#\eus P(\sfr)$. 
If $\sfr$ is \textsf{wmf}, then $\eus P(\sfr)$ is said to be a \textsf{wmf}-{\it poset\/}.
Clearly, one can  define weight posets and 
\textsf{wmf}-representations for arbitrary reductive Lie algebras. 
However, if $\sfr$ is a simple $\g$-module, then the center of $\g$ does not affect the property of being \textsf{wmf}. 

In this article, we study \textsf{wmf}-posets;
specifically, we are interested in relations between $\dim\sfr$ and the number of edges,
$\#\eus E(\sfr)$. The irreducible \textsf{wmf}-representations of {\sl simple\/} Lie algebras 
are classified by R.\,Howe~\cite[4.6]{howe}.
We begin with computing the number of edges for all representations in Howe's list.
It is then easy to get formulae for the irreducible  \textsf{wmf}-representations
of {\sl semisimple\/} algebras.
We also observe that there are non-trivial isomorphisms between weight posets of different
irreducible \textsf{wmf}-representations. Therefore, the number of different \textsf{wmf}-posets
is considerably smaller than that of \textsf{wmf}-representations.

Our main results concern \textsf{wmf}-posets associated with  gradings of simple Lie 
algebras.
We consider two types of gradings: 
\begin{equation*}
\text{$\BZ$-grading:}   \quad \g=\bigoplus_{i\in\BZ}\g(i), \qquad
\text{periodic or $\BZ_m$-grading:}   \quad \g=\bigoplus_{i\in\BZ_m}\g_i .
\end{equation*}

\noindent 
Here $\g(0)$ and $\g_0$ are reductive Lie algebras.
For $\BZ$-gradings, $\rk\g=\rk\g(0)$ and each $\g(i)$ is a \textsf{wmf} $\g(0)$-module. 
For periodic gradings, it is not always the case that $\rk\g=\rk\g_0$. However, if 
$\rk\g=\rk\g_0$, which means that the corresponding
periodic automorphism $\vartheta\in{\rm Aut}(\g)$ is inner, 
then each $\g_i$ is a \textsf{wmf} $\g_0$-module.
After work of Vinberg \cite{vi76}, it is known that the representations $(\g(0):\g(i))$ and
$(\g_0:\g_i)$ have nice invariant-theoretic properties. (Actually, it suffices to consider 
the representations with $i=1$.) 
Let $\eus E(i)$ (resp. $\eus E_i$) denote the set of edges in the Hasse diagram of $\eus P(\g(i))$ (resp. $\eus P(\g_i)$). Our main result can be regarded as combinatorial 
manifestation of `niceness' of representations considered by Vinberg.

\begin{thm}   \label{thm:intro1} Let $\g$ be a simple Lie algebra. 

1) In case of $\BZ$-gradings, we always have $0< 2\dim\g(1)-\#(\eus E(1))\le h$, where
$h$ is the Coxeter number of\/ $\g$.
Furthermore,  if $m=\max\{j\mid \g(j)\ne 0 \}>1$, then $2\dim\g(1)-\#\eus E(1)< h$. 

2)  For periodic gradings of inner type, we have $0\le 2\dim\g_1-\#(\eus E_1)$. The equality 
$2\dim\g_1=\#(\eus E_1)$ holds if and only if $\g$ is simply-laced and $\g_0$ is semisimple.
\end{thm}

\noindent It is worth noting that $\g_0$ is semisimple if and only if $\g_1$ is a simple
$\g_0$-module (Vinberg~\cite{vi76}). 

The most interesting constraint in the theorem is that $\#\eus E(1)/\dim\g(1)<2$
(or $\le 2$ for periodic gradings). This is a real condition, since the ratio 
$\#\eus E(\sfr)/\dim\sfr$ can be arbitrarily large even for irreducible \textsf{wmf}-representations $\sfr$ of simple Lie algebras.

Most of the proofs are based on 
case-by-case considerations. 
Some {\sl a priori\/} proofs occur in the simply-laced case.
Let us say that a $\BZ$-grading is {\it standard\/} if $\g(-1)\oplus\g(0)\oplus\g(1)$ generate 
$\g$ as Lie algebra.
(Another, but equivalent definition is given in Section~\ref{sect:wmf-z}.)

\begin{thm}   \label{thm:intro2}
Let $\g=\bigoplus_{i\in\BZ}\g(i)$ be a standard $\BZ$-grading. If $\g$ is simply-laced
and $\rk [\g(0),\g(0)]=\rk\g-k$, then 
$\sum_{i\ge 1} (2\dim\g(i)-\#\eus E(i))=k{\cdot}h$.
\end{thm}

\noindent 
There is also a partial converse to Theorem~\ref{thm:intro1} which is valid in the 
simply-laced case, see Theorem~\ref{thm:posets-2v>=p}.

Yet another property of \textsf{wmf}-posets associated with $\BZ$-gradings
is expressed in terms of (upper) covering polynomials $\eus K(t)$~\cite{coveri}, see the definition in \S\ref{1.1}.
We show that,  for the posets $\eus P(\g(i))$,  $\deg\eus K(t)$ is at most $3$.
Because we compute the upper covering polynomials for all \textsf{wmf}-posets,
this provides another necessary combinatorial condition for a  \textsf{wmf}-representation to 
occur in connection with a $\BZ$-grading. 
The degree bound also yields a simple interpretation
of inequality $0< 2\dim\g(1)-\#(\eus E(1))$ in terms of coefficients of $\eus K(t)$.

Here is a brief description of the article. In \S\,\ref{sec:general}, we fix main notation and recall some results 
on the poset of positive roots from \cite{jac06, coveri}. In \S\,\ref{sect:numb-edg}, we compute the number of edges for the \textsf{wmf}-posets and point out \textsf{wmf}-representations 
with isomorphic weight posets. In \S\,\ref{sect:wmf-z} and \ref{sec:periodic}, we consider
\textsf{wmf}-posets associated with $\BZ$- and periodic gradings, respectively.
The upper covering polynomials of \textsf{wmf}-posets are discussed in \S\,\ref{sec:cover}.

\vskip1ex
{\small {\bf Acknowledgements.}  
This work was done during my 
stay at the Max-Planck-Institut f\"ur Mathematik (Bonn). 
I am grateful to this institution for the warm hospitality and support.}

\section{Generalities}   \label{sec:general}

\subsection{Posets, edges, Hasse diagrams} \label{1.1}
 Let $(\eus P, \succcurlyeq)$ be a finite 
poset (partially ordered set). We say that $\mu$ {\it covers\/} $\nu$ if  $\mu\ne\nu$,
$\mu\succcurlyeq\nu$, and if $\gamma$ satisfies $\mu\succcurlyeq\gamma\succcurlyeq\nu$,
then either $\gamma=\mu$ or $\gamma=\nu$.
The {\it Hasse diagram\/} of $\eus P$ is  the directed graph  
whose set  of vertices is $\eus P$, and the edges are the pairs  $(\mu,\nu)\in 
\eus P\times \eus P$ such that  $\mu$ covers $\nu$. For brevity, 
such a pair $(\mu,\nu)$ will also be referred to as an edge of $\eus P$. We say that 
$\eus P$ is  {\it connected\/} if the Hasse diagram of $\eus P$ is.

In \cite{coveri}, we considered  two statistics on a finite poset $\eus P$ and thereby two
generating functions. Namely, for $a\in\eus P$, one can count the number of elements that 
either are covered by $a$ or covers $a$. 
In particular, 
set $\eus P^{(j)}=\{ a\in \eus P \mid \text{$a$ covers  $j$ elements of $\eus P$}\}$.
Then the {\it upper covering polynomial\/}  of
$\eus P$ is
 \[
     \eus K_{\eus P}(t)= \sum_{j\ge 0} \# (\eus P^{(j)}) t^j .
 \]
Notice that $\eus K_{\eus P}(1)=\#\eus P$ and $\eus K'_{\eus P}(1)=\# \eus E(\eus P)$,  the number of edges of $\eus P$.
The related notion of the {\it lower covering polynomial\/} is not needed in this article, 
since these two polynomials coincide for the weight  posets.

\subsection{Root systems and weight posets}

Let $\g$ be a reductive algebraic Lie algebra of semisimple rank $n$.
Fix a triangular decomposition $\g=\n^-\oplus\te\oplus\n^+$, where $\te$ is a Cartan subalgebra.
Associated with this choice, one obtains 
\begin{itemize}
\item the set of positive roots $\Delta^+$;
\item the set of simple roots $\Pi=\{\ap_1,\dots,\ap_n\}\subset \Delta^+$;
\item the set of dominant weights $\mathcal X_+$ and 
of fundamental weights $\{\varpi_1,\dots.\varpi_n\}$.
\end{itemize}
For $\gamma\in\Delta^+$, $[\gamma:\ap_i]$ is the coefficient of $\ap_i$ 
in the expression of $\gamma$ via the simple roots.
The {\it height\/} of $\gamma$ is $\hot(\gamma)=\sum_{i=1}^n [\gamma:\ap_i]$.
If $\g$ is simple, then $h=h(\g)$ is the {\it Coxeter number\/} and 
$\theta$ is the highest  root of $\Delta^+$. 
Recall that $\hot(\theta)
=h-1$. We use the numbering of simple roots as in \cite[Tables]{t41}.

We endow $\Delta^+$ with  the usual {\it root order\/}.
This means that $\gamma$ covers $\beta$ if and only if $\gamma-\beta\in\Pi$.
If  $(\gamma,\beta)$ is an edge 
of $\Delta^+$ and $\gamma-\beta=\ap_i$, then this edge is said to be 
{\it of type\/} $\ap_i$. The set of all edges is denoted by $\eus E(\Delta^+)$ and 
the subset of edges of type $\ap_i$ is denoted by $\eus E_{\ap_i}(\Delta^+)$.
For $\g$ non-simple,  the Hasse diagram of $\Delta^+$ 
is the disjoint union of the Hasse diagrams of simple ideals of $\g$. 
Therefore, it suffices to understand the structure of Hasse diagram if $\g$ is simple.
We say that $\g$ (or $\Delta$) is {\it simply-laced\/} if so is the Dynkin diagram of $\Delta$.

\begin{thm}[\protect{\cite[Sect.\,1]{jac06}}]    \label{thm:edges06}
Suppose $\g$ is simple.
The number of edges of type $\ap_i$ depends only on  length of $\ap_i$. 
If $\Delta$ is simply laced, then $\#(\eus E_{\ap_i}(\Delta^+))=h-2$.
In particular, $\#\eus E(\Delta^+)=n(h-2)$.
\end{thm}

\begin{rmk}  \label{rmk:edges-non}
If $\Delta$ is multiply-laced and $\ap_i$ is long, then  $\#(\eus E_{\ap_i}(\Delta^+))=h^*-2$,
where $h^*$ is the {\it dual Coxeter number}. For $\ap_i$ short, the formula (and the 
proof!) is not so nice, see \cite[Theorem\,1.2]{jac06}.
\end{rmk}

For a finite-dimensional $\g$-module $\sfr$,
let $\eus P(\sfr)$ be the set of $\te$-weights of $\sfr$.
Again, we regard $\eus P(\sfr)$ 
as poset with respect to the { root order\/} and call it the {\it weight poset\/} of $\sfr$.
The set of all edges of $\eus P(\sfr)$ (resp. edges of type $\ap_i$)
is denoted by  $\eus E(\sfr)$ (resp. $\eus E_{\ap_i}(\sfr)$).
If $\sfr=\sfr(\lb)$ is the simple  $\g$-module with highest weight $\lb\in\mathcal X_+$,
then we write $\eus P(\lb)$ and $\eus E(\lb)$ for the sets of weights and edges, respectively. 
Note that $\lb$ is the unique maximal element of $\eus P(\lb)$, and the lowest weight of
$\sfr(\lb)$ is the unique minimal element.
If  we wish  to stress the dependance of either
of the previous objects on $\g$, then we write $\Pi(\g)$ or
$\sfr(\g,\lb)$ 
or $\eus E(\g,\lb)$, etc.  Recall that $\eus P(\lb)$ is called a \textsf{wmf}-{\it poset\/} if $\sfr(\lb)$ is 
 \textsf{wmf} $\g$-module.

\begin{thm}[\protect{\cite[Theorem\,2.2]{jac06}}]   \label{thm:edges-wmf06}
Suppose $\g$ is simple and $\eus P(\lb)$ is a \textsf{wmf}-poset.
Then $\#(\eus E_{\ap_i}(\lb))$ depends only on  length of $\ap_i$.
In particular, if $\Delta$ is simply-laced, then $\rk\g$ divides $\#(\eus E(\lb))$.
\end{thm}

The following property of  upper covering  polynomials 
is proved in \protect{\cite[Sect.\,2]{coveri}}.
 
\begin{thm}
\label{thm:deg-coveri}
For any root system $\Delta^+$, we have
$\deg \eus K_{\Delta^+} \le 3$, and $\deg \eus K_{\Delta^+} = 3$ if and only if
$\Delta$ is of type $\GR{D}{n}$ or $\GR{E}{n}$ or $\GR{F}{4}$.
Furthermore, suppose that $\gamma$ covers three other roots,
i.e., $\gamma-\ap_{i_j}\in\Delta^+$ for some  $\ap_{i_1},\ap_{i_2},\ap_{i_3}\in\Pi$.
Then these simpe roots  are pairwise orthogonal.
\end{thm}

\section{On edges of weight posets of \textsf{wmf}-representations}
\label{sect:numb-edg}

\noindent
In this section, we compute the number of edges in (the Hasse diagrams of) connected 
\textsf{wmf}-posets and describe some isomorphisms between  
weight posets.

For the  simple Lie algebras,
the list of all irreducible \textsf{wmf}-representations is obtained by R.\,Howe
\cite[4.6]{howe}. 
This information is contained in the first two columns of Table~\ref{tabl} (using our convention
on the numbering of $\varpi_i$).

The computation of the number of edges is simplified by the fact that different weight 
posets  can naturally be isomorphic. Clearly, the weight poset does not
change, if we replace $\sfr(\lb)$ with its dual (use the longest element of the Weyl group). 
Two important non-trivial isomorphisms
are described below.

\begin{thm}    \label{thm:isom-posets} \ \phantom{\,}  \nopagebreak

1) $\eus P(\GR{A}{n},m\varpi_1)\simeq \eus P(\GR{A}{n+m-1},\varpi_n)\simeq \eus P(\GR{A}{n+m-1},\varpi_m)\simeq \eus P(\GR{A}{m},n\varpi_1)$.
In particular, for $m=2$ we have
 \ $\eus P(\GR{A}{2},n\varpi_1)\simeq 
\eus P(\GR{A}{n},2\varpi_1) \simeq \eus P(\GR{A}{n+1},\varpi_2)$

2) $\eus P(\GR{B}{n},\varpi_n) \simeq \eus P(\GR{D}{n+1},\varpi_{n+1})$.
\end{thm}\begin{proof}
1) Since $\sfr(\GR{A}{n+m-1},\varpi_n)^*=\sfr(\GR{A}{n+m-1},\varpi_m)$, 
it suffices to establish the first isomorphism.

Let $(e_1,\dots, e_{n+m})$ be the standard weight basis of the tautological representation,
$\sfr(\varpi_1)$, 
of $\GR{A}{n+m-1}=\mathfrak{sl}_{n+m}$. That is, the weight of $e_i$ is $\esi_i$,
and $\sum_i \esi_i=0$. The simple roots are $\ap_i=\esi_i-\esi_{i+1}$.
The polyvectors $e_{i_1}\wedge\dots\wedge e_{i_n}$, $1\le i_1< \ldots < i_n \le n+m$,
form a weight basis for $\wedge^n(\sfr(\varpi_1))=\sfr(\varpi_n)$, and 
the weight of $e_{i_1}\wedge\dots\wedge e_{i_n}$ is $\esi_{i_1}+\ldots +\esi_{i_n}$.
The above polyvector is identified with the sequence $\boldsymbol{i}=(i_1,\dots,i_n)$.
It will be convenient to assume that $i_0=0$ and $i_{n+1}=n+m+1$.

On the other hand, let $x_1,\dots,x_{n+1}$ be the standard weight basis of the 
tautological representation, $\sfr(\varpi_1)$, 
of $\GR{A}{n}=\slno$.  Then the monomials
$x_1^{j_1}\ldots x_{n+1}^{j_{n+1}}$, $j_1+\ldots +j_{n+1}=m$, form a weight basis for
$\mathcal S^m(\sfr(\varpi_1))=\sfr(m\varpi_1)$.
The above monomial has weight $\sum_k j_k\esi_k$ and 
is identified with the sequence $\boldsymbol{j}=(j_1,\dots,j_{n+1})$.

Define a correspondence between the two weight bases as follows:
\begin{equation}  \label{eq:biject}
   (i_1,\dots,i_n) \mapsto (n+m-i_n, i_n - i_{n-1}-1,\dots, i_2-i_1-1, i_1-1) .
\end{equation}
In other words, $j_k=i_{n+2-k}-i_{n+1-k}-1$, \ $k=1,2,\dots,n+1$.
Clearly,  this  is a bijection. Let us verify that this also 
provides an isomorphism of weight posets.

For $\nu=\esi_{i_1}+\ldots +\esi_{i_n}\in\eus  P(\varpi_n)$ and 
$\ap_j\in \Pi(\GR{A}{n+m-1})$, one has $\nu-\ap_j
\in \eus  P(\varpi_n)$ if and only if $j=i_k$ and $i_{k+1}-i_k\ge 2$ for some $k$.
On the level of $\boldsymbol{i}$-sequences, this means that  
$i_k$ is replaced with $i_k+1$, while all
other components remain intact. 
In terms of the corresponding $\boldsymbol{j}$-sequences, we make transformation
$j_{n-k+1}\mapsto j_{n-k+1}-1$ and $j_{n-k+2}\mapsto j_{n-k+2}+1$, which corresponds
to subtraction the simple root $\ap_{n-k+1}\in \Pi(\GR{A}{n})$.
Therefore, mapping \eqref{eq:biject} yields a bijection between two sets of edges.

2)  The weights of $\sfr(\GR{B}{n},\varpi_n)$ are $(\pm\esi_1\pm\esi_2\ldots \pm\esi_n)/2$,
where all combinations of signs are allowed.
The weights of $\sfr(\GR{D}{n+1},\varpi_{n+1})$ are $(\pm\esi_1\pm\esi_2\ldots 
\pm\esi_{n+1})/2$, where the total number of minuses is even.
The bijection between the weights is quite obvious.
For $\nu\in\eus P(\GR{B}{n},\varpi_n)$, the corresponding weight
for $\GR{D}{n+1}$ is $\nu'=\nu\pm (\esi_{n+1}/2)$, where the sign of 
$\esi_{n+1}$ is determined by the condition that the total number of ``$-$'' in
$\nu'$ to be even.

We also use `prime' to mark simple roots of $\GR{D}{n+1}$.
Recall that 
\[ 
 \Pi(\GR{B}{n}){=}\{\esi_1-\esi_{2},\dots,\esi_{n-1}-\esi_{n},\esi_n\}, \quad
 \Pi(\GR{D}{n+1}){=}\{\esi_1-\esi_{2},\dots,\esi_{n-1}-\esi_{n}, \esi_n-\esi_{n+1},
 \esi_n+\esi_{n+1}\}.
\]

It follows that,
for $1\le i\le n-1$, the edges  of type $\ap_i$ are the ``same'' in both 
posets; i.e., $\nu-\ap_i\in \eus P(\GR{B}{n},\varpi_n)$ if and only if 
$\nu'-\ap'_i\in \eus P(\GR{D}{n},\varpi_{n+1})$. Something not entirely trivial only happens
for $\ap_n$.  If $\nu-\ap_n\in \eus P(\GR{B}{n},\varpi_n)$, then the last sign in $\nu$ must be 
`plus'. Now, we have a fork. If $\nu'=\nu+ (\esi_{n+1}/2)$, then 
$\nu'-\ap'_{n+1}\in \eus P(\GR{D}{n},\varpi_{n+1})$.
If $\nu'=\nu- (\esi_{n+1}/2)$, then 
$\nu'-\ap'_{n}\in \eus P(\GR{D}{n},\varpi_{n+1})$.
\end{proof}

\begin{thm}    \label{thm:number-edges} \ \phantom{\,}

1. $\#\eus E(\GR{A}{n},\varpi_m)=
m\genfrac{(}{)}{0pt}{}{n}{m}=n\genfrac{(}{)}{0pt}{}{n-1}{m-1}$;

2. $\#\eus E(\GR{D}{n},\varpi_n)=n2^{n-3}$.
\end{thm}
\begin{proof}
1.  Let $\esi_1,\dots,\esi_{n+1}$  be the weights of the tautological representation
of $\GR{A}{n}$.
Then the weights of $\sfr(\varpi_m)$ are $\esi_{i_1}+\ldots +\esi_{i_m}$,
where $1\le i_1< i_2< \ldots < i_{m}\le n+1$ (cf. the proof of Theorem~\ref{thm:isom-posets}).
Let $\nu\in\eus P(\varpi_m)$ be a weight such that $\nu-\ap_1\in \eus P(\varpi_m)$.
Then $\nu$ must be  of the form 
$\esi_1+\esi_{i_2}+\ldots +\esi_{i_m}$ with $i_2\ge 3$. Obviously, the number of such 
weights $\nu$ is equal to $\genfrac{(}{)}{0pt}{}{n-1}{m-1}$, and this is the number of edges
of type $\ap_1$.
Making use of Theorem~\ref{thm:edges-wmf06}, we conclude that
the total number of edges is $n\genfrac{(}{)}{0pt}{}{n-1}{m-1}$.

2. 
 Let $\nu\in \eus P(\varpi_n)$ be a weight such that $\nu-\ap_1\in \eus P(\varpi_n)$.
Then $\nu$ must be  of the form 
$(\esi_1-\esi_2\pm\esi_3\ldots \pm\esi_n)/2$. Since the number of minuses is supposed to be 
even, there are $2^{n-3}$ possibilities for such $\nu$, and hence  $2^{n-3}$ edges
of type $\ap_1$. Making use of Theorem~\ref{thm:edges-wmf06}, we conclude that
the total number of edges is $n2^{n-3}$.  
\end{proof}

\noindent
For the simplest representations of $\GR{B}{n},\,\GR{C}{n}$, and $\GR{G}{2}$
(i.e., for $\lb=\varpi_1$), $\eus P(\lb)$ is a chain. Hence $\#\eus E(\lb)=\dim\sfr(\lb)-1$.
This also provides the poset isomorphisms:
\[
   \eus P(\GR{B}{n},\varpi_1)\simeq \eus P(\GR{A}{2n},\varpi_1), \quad
   \eus P(\GR{C}{n},\varpi_1)\simeq \eus P(\GR{A}{2n-1},\varpi_1),\quad
   \eus P(\GR{G}{2},\varpi_1)\simeq \eus P(\GR{A}{6},\varpi_1).
\]
The Hasse diagram of $\eus P(\GR{D}{n},\varpi_1)$ is 

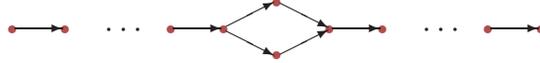
\begin{center}
\begin{picture}(240,30)(0,-3)
\multiput(10,10)(20,0){2}{\color{MIXT}\circle*{3}}
\multiput(70,10)(20,0){2}{\color{MIXT}\circle*{3}}
\multiput(130,10)(20,0){2}{\color{MIXT}\circle*{3}}
\multiput(190,10)(20,0){2}{\color{MIXT}\circle*{3}}
\multiput(110,0)(0,20){2}{\color{MIXT}\circle*{3}}

\multiput(11,10)(60,0){4}{\vector(1,0){18}}

\multiput(91,11)(20.5,-10.5){2}{\vector(2,1){18}}
\multiput(91,9)(20.5,10.5){2}{\vector(2,-1){18}}

\put(45,7){$\cdots$}
\put(165,7){$\cdots$}
\end{picture}
\end{center}

\noindent
The remaining cases can be handled 
directly or using some theory from the next section.
Our computations are gathered in the table. The last column will be needed in Section~\ref{sec:periodic}.

\begin{table}[htb]   
\begin{center}
\caption{The number of edges  for the \textsf{wmf} representations}
\begin{tabular}{cccccc}  \label{tabl}
   $\g$  & $\lb$  & $\dim\sfr(\lb)$ & $\#\eus E(\lb)$ &  Reference &  $\#\eus E(\lb)/\dim\sfr(\lb)$
\\ \hline\hline
\smash{\raisebox{-2.5ex}{$\GR{A}{n}$}}   & $\varpi_m$  & $\displaystyle \genfrac{(}{)}{0pt}{}{n+1}{m}$
 & $\displaystyle m\genfrac{(}{)}{0pt}{}{n}{m}$
& Theorem~\ref{thm:number-edges} & $\frac{m(n+1-m)}{n+1}$
\\ 
     & $ m\varpi_1, m\varpi_n $  & $\displaystyle \genfrac{(}{)}{0pt}{}{n+m}{m}$  \rule{0pt}{4.5ex}
   & $ \displaystyle m\genfrac{(}{)}{0pt}{}{n+m-1}{m}$ & 
     Theorems~\ref{thm:isom-posets} \& \ref{thm:number-edges}  & $\frac{nm}{n+m}$
\\ \hline
\smash{\raisebox{-2ex}{$\GR{B}{n}$}} & $\varpi_n$ & $2^n$ \rule{0pt}{2.5ex}
& $(n+1)2^{n-2}$ & Theorems~\ref{thm:isom-posets} \& \ref{thm:number-edges} 
&  $\frac{n+1}{4}$
\\ 
            & $\varpi_1$  & $2n+1$ & $2n$  & $\eus P(\lb)$ is chain & $1-\frac{1}{2n+1}$ \\ \hline
\smash{\raisebox{-2ex}{$\GR{C}{n}$}} & $\varpi_1$ & $2n$ \rule{0pt}{2.5ex}
& $2n-1$  & $\eus P(\lb)$ is chain & $1-\frac{1}{2n}$
\\ 
            & $\varpi_3\ (n=3)$ & $14$ & $17$ & directly & $17/14$
                      \\ \hline
\smash{\raisebox{-2ex}{$\GR{D}{n}$}} & $\varpi_1$ & $2n$ & $2n$ & see figure & $1$
\\ 
    & $\varpi_{n-1},\varpi_n$ & $2^{n-1}$ & $n2^{n-3}$ & Theorem~\ref{thm:number-edges} 
&  $\frac{n}{4}$
\\ \hline
$\GR{E}{6}$  & $\varpi_1$  & 27 & 36 & Example~\ref{ex:e6-vp1} & $4/3$ \\ \hline
$\GR{E}{7}$  & $\varpi_1$  & 56 & 84 & Example~\ref{ex:Z-special} & $3/2$\\ \hline
$\GR{G}{2}$  & $\varpi_1$  & 7 & 6 & $\eus P(\lb)$ is chain  & $7/6$
\\ \hline
\end{tabular}
\end{center}
\end{table}

\vskip1ex\noindent
Below, we list  all weight posets occurring in Table~\ref{tabl}, up to isomorphism:
\begin{align}  \label{align:posets}
& \text{Serial cases:
$\eus P(\GR{A}{n}, \varpi_m)$, $1\le m\le (n+1)/2$; \ 
$\eus P(\GR{D}{n}, \varpi_n)$, $n\ge 5$; \ 
$\eus P(\GR{D}{n}, \varpi_1)$, $n\ge 4$;}  \\  \notag
& \text{Sporadic cases: $\eus P(\GR{E}{6}, \varpi_1)$, $\eus P(\GR{E}{7}, \varpi_1)$, 
$\eus P(\GR{C}{3}, \varpi_3)$;} 
\end{align}

\vskip1ex\noindent
The irreducible \textsf{wmf} representations of a semisimple algebra $\es$
are tensor products of irreducible \textsf{wmf} representations of different simple ideals of
$\es$. It is  therefore easy to compute the number of edges for them using  information on
the simple ideals.

\begin{lm}   \label{lm:tensor-wmf}
Let $(\g', \sfr')$, $(\g'',\sfr'')$ be two  \textsf{wmf\/} representations.
Then $(\g'\times\g'', \sfr'\otimes\sfr'')$ is also \textsf{wmf\/} and
\[
    \#\eus E(\sfr'\otimes\sfr'')=\dim\sfr'\cdot \#\eus E(\sfr'')+\dim\sfr''\cdot \#\eus E(\sfr') .
\]
\end{lm}\begin{proof}
The Hasse diagram of $\eus P(\sfr'\otimes\sfr'')$ is determined by the following conditions:
\begin{itemize}
\item \  $\eus P(\sfr'\otimes\sfr'')=\eus P(\sfr')\times \eus P(\sfr'')$;
\item   Suppose $a_1,a_2 \in  \eus P(\sfr')$ and $b_1,b_2 \in  \eus P(\sfr'')$. Then 
$((a_1,b_1),(a_2,b_2))\in \eus  E(\sfr'\otimes\sfr'') $ if and only if either 
$a_1=a_2$ and $(b_1,b_2)\in \eus E(\sfr'')$ or $b_1=b_2$ and 
$(a_1,a_2)\in \eus E(\sfr')$.
\end{itemize}
Therefore $\eus E(\sfr'\otimes\sfr'') \simeq (\eus E(\sfr')\times \eus P(\sfr'')) \sqcup 
(\eus E(\sfr'')\times \eus P(\sfr'))$.
\end{proof}

\noindent In terminology of Graph Theory, the Hasse diagram of $\eus P(\sfr'\otimes\sfr'')$
is called the {\it cartesian product\/} of Hasse diagrams of  $\eus P(\sfr')$ and $\eus P(\sfr'')$.
We also apply this term to the posets themselves.

\begin{cl}  \label{cor:additive} \quad
$\displaystyle\frac{\#\eus E(\sfr'\otimes\sfr'')}{\dim (\sfr'\otimes\sfr'')}=
\frac{\#\eus E(\sfr')}{\dim (\sfr')}+\frac{\#\eus E(\sfr')}{\dim (\sfr')} .
$
\end{cl}

\noindent
We say that two \textsf{wmf}-representations are {\it equivalent\/} 
if the corresponding weight posets are isomorphic. It follows from the preceding discussion that if one of the factors in $(\g'\times\g'', \sfr'\otimes\sfr'')$ is being replaced with an 
equivalent one,
then the resulting tensor products are equivalent.

\section{Weight posets for \textsf{wmf}-representations associated with
$\BZ$-gradings}   \label{sect:wmf-z}

\noindent
In this section,  $\g$ is a simple Lie algebra.
Let $\g=\bigoplus_{i\in\BZ} \g(i)$ be a $\BZ$-grading.
Then there is a semisimple $s\in\g$ such that $\g(i)=\{x\in\g\mid [s,x]=ix\}$.
Consequently,  $\g(0)$ is reductive,
$\rk\g(0)=\rk\g$, and each $\g(i)$ is a \textsf{wmf} $\g(0)$-module.
Our objective is to prove that such \textsf{wmf}-representations possess a special
property.
Without loss of generality, we may assume that $s\in \te$ and $\ap_i(s)\ge 0$ for all 
$\ap_i\in\Pi$. Then $\te\subset\g(0)$, and $\g(0)$ inherits the triangular decomposition from
$\g$. 

By \cite[\S\,1.2, \S\,2.1]{vi76}, if one is interested in possible $\g(0)$-modules $\g(i)$, 
then it is enough to consider the $\g(0)$-modules $\g(1)$ for all simple $\g$. 
(For $i >1$,  the question is reduced to considering the induced grading of
a certain simple subalgebra of $\g$.)
For this reason, it suffices to consider $s\in\te$
such that $\ap_i(s)\in\{0,1\}$. 
Then $\Pi=\Pi(0)\sqcup\Pi(1)$, where $\Pi(i)=\{\ap\in \Pi\mid \ap(s)=i\}$.
The corresponding $\BZ$-gradings are said to be {\it standard}.
More precisely, if $\#\Pi(1)=k$, then we call it a $k$-{\it standard\/} grading.
A standard $\BZ$-grading will be represented by the Dynkin diagram of $\g$,
where the vertices in $\Pi(1)$ are coloured.
If $\Pi(1)=\{\ap_{i_1},\dots,\ap_{i_k}\}$, then the $\ap_{i_j}$'s are precisely 
the lowest weights of the simple $\g(0)$-submodules in $\g(1)$.
Therefore, $\g(1)$ is the sum of $k$ simple $\g(0)$-modules.

Let $\Delta(i)$ denote the set of roots ($\te$-weights) in $\g(i)$. Then $\Delta(0)$ is the root
system of $\g(0)$ and  $\Pi(0)$ is the set of simple roots in $\Delta(0)^+$.
We regard $\Delta(i)$ as weight poset of the $\g(0)$-module $\g(i)$ and 
write $\eus E(i)$ for the set of edges
in (the Hasse diagram of) $\Delta(i)$. 
If  $m=\max\{j\mid \g(j)\ne 0\}$, then
$\theta\in\Delta(m)$.

\begin{rmk}
The Hasse diagrams of posets $\Delta(i)$ are obtained as follows. Take  the Hasse 
diagram of $\Delta^+$
and remove all the edges of types $\ap_{i_1},\dots,\ap_{i_k}$. The remaining 
(disconnected) graph is the union of Hasse diagrams of $\Delta(i)$, $i\ge 1$, 
and $\Delta(0)^+$. 
\end{rmk}

\begin{thm}   \label{thm:2v-e=h}
Let $\g=\bigoplus_{i\in\BZ}\g(i)$ be a $k$-standard grading. 
Suppose $\g$ is simply laced. Then
\begin{equation}  \label{eq:iz-teor}
\sum_{i\ge 1} (2\dim\g(i)-\#\eus E(i))=k{\cdot}h .
\end{equation}
\end{thm}
\begin{proof}
Recall that $\dim\n^+=nh/2$. Let $\{n_j,h_j\}_{j\in J}$ be the ranks and Coxeter numbers,
respectively, for all simple ideals of
$\g(0)$. Then $\dim (\g(0)\cap\n^+)=\sum_{j\in J} n_jh_j/2$. Note that $\sum_{j\in J} n_j=n-k$.
It follows that  
\begin{equation}  \label{eq:2v}
2\sum_{i\ge 1}\dim\g(i) =nh- \sum_{j\in J} n_j h_j=kh+\sum_{j\in J} n_j(h-h_j) \ .
\end{equation}
If $\Pi(1)=\{\ap_{i_1},\dots,\ap_{i_k}\}$, then 
the posets $\Delta(i)$, $i\ge 1$, do not contain edges of types 
$\ap_{i_1},\dots,\ap_{i_k}$ and each simple ideal of $\g(0)$ is also simply-laced. Therefore
\begin{multline}  \label{eq:e}
  \sum_{i\ge 1}\#\eus E(i)=\#\eus E(\Delta^+)-\#\eus E(\Delta(0)^+)- \sum_{s=1}^k
  \#\{\text{the edges of type $\ap_{i_s}$}\}= \\
  =n(h-2)-\sum_{j\in J} n_j(h_j-2)-kh=\sum_{j\in J} n_j(h-h_j) .
\end{multline}
Taking the difference of Eq.~\eqref{eq:2v} and  \eqref{eq:e} yields the assertion.
\end{proof}

\begin{rmk}
In the multiply-laced case, there is no such a nice formula. The reason is that the
total number of edges in both $\Delta^+$ and $\Delta(0)^+$ does not admit a simple
closed expression, see Remark~\ref{rmk:edges-non}. Furthermore, 
$\sum_{i\ge 1} (2\dim\g(i)-\#\eus E(i))$ can be more than
$k{\cdot}h$ and it does not depend only on $\#\Pi(1)$, see example below.
\end{rmk}

\begin{ex}   \label{ex:f4}
For $\g=\GR{F}{4}$, consider all 1-standard $\BZ$-gradings (i.e., with
$k=1$). In these four cases, the left hand side of Eq.~\eqref{eq:iz-teor} equals:
\begin{tabular}{c|c|c|c}
$\ap_1$  & $\ap_2$ & $\ap_3$ & $\ap_4$ \\ \hline
        16   &        18   &      15     & 13 
\end{tabular}. \ 
In particular, each sum is larger than $h=12$.
\end{ex}
\noindent
The subspace $\p=\bigoplus_{i\ge 0} \g(i)$ is a parabolic subalgebra,
and it is maximal parabolic  if and only the grading is 1-standard.
Then each $\g(i)$ is a simple $\g(0)$-module \cite[3.5]{t41}.
Conversely, if $\#\Pi(1)>1$, then
$\g(1)$ is not a simple module. Without loss of generality, we may restrict ourselves 
with $1$-standard $\BZ$-gradings.
In other words, every {\sl simple\/} $\g(0)$-submodule 
of $\g(1)$ can be obtained as the component of degree 1
of a simple subalgebra of $\g$ endowed with induced grading.
Indeed, suppose that $\#\Pi(1)>1$ and $\ap_j\in \Pi(1)$. Let $\sfr$ be the simple 
$\g(0)$-submodule of $\g(1)$ with lowest weight $\ap_j$.
Let us remove the vertices 
$\Pi(1)\setminus \{\ap_j\}$ from the Dynkin diagram and take the connected subgraph
containing $\ap_j$.  This subgraph is the Dynkin diagram of a simple subalgebra $\es
\subset\g$. The vertex $\ap_j$ determines a $\BZ$-grading $\es=\bigoplus_{i\in\BZ} \es(i)$,
and 
$\sfr$ is isomorphic to the $\es(0)$-module $\es(1)$.
If $\Pi(1)=\{\ap_i\}$, then $\max\{j\mid \g(j)\ne 0\}=[\theta:\ap_i]$.
For $[\theta:\ap_i]=1$, we have $\g=\g(-1)\oplus\g(0)\oplus\g(1)$, and 
the nilpotent radical of $\p$ is abelian. 
 Such $\BZ$-gradings are said to be {\it short}.

\begin{thm}    \label{thm:2v1-e1=h}
For any short grading, one  has \ $2\dim\g(1)-\#\eus E(1)=h$.
\end{thm}\begin{proof}
If $\Delta$ is simply-laced, then the assertion follows from Theorem~\ref{thm:2v-e=h}.
If $\Delta$ is multiply-laced, then we have only two possibilities:
$\ap_n$ for $\GR{C}{n}$ and $\ap_1$ for $\GR{B}{n}$, and everything can be counted directly.

For $(\GR{C}{n},\ap_n)$, 
he semisimple part of $\g(0)$, denoted $\g(0)'$, is $\GR{A}{n-1}$, and the 
$\GR{A}{n-1}$-module
$\g(1)$ is $\sfr(2\varpi_1)$. Here $\dim\g(1)=n(n+1)/2$ and $\#\eus E(2\varpi_1)=
\#\eus E(1)=n(n-1)$, see Table~\ref{tabl}. 
The case of $\GR{B}{n}$ is left to the reader.
\end{proof}

\begin{ex}  \label{ex:e6-vp1}
If $\g=\GR{E}{7}$, then $[\theta:\ap_1]=1$. For the respective short grading, we have
$\g(0)'=\GR{E}{6}$ and $\g(1)$ is the $\GR{E}{6}$-module $\sfr(\varpi_1)$, of dimension
$27$.
Therefore $\#\eus E(\varpi_1)=2\dim \sfr(\varpi_1) -h(\g)=36$.
\end{ex}

\begin{ex}  \label{ex:Z-special}
Every simple Lie algebra has a unique (up to conjugation) $\BZ$-grading 
$\g=\bigoplus_{-2\le i\le 2}\g(i)$ such that  $\dim\g(2)=1$. (The grading corresponding
to the minimal nilpotent orbit.) Then 
$\Delta(2)=\{\theta\}$ and $\Delta(i)=\{\gamma\mid (\gamma,\theta^\vee)=i\}$.
Obviously, $\eus E(2)=\varnothing$. Suppose $\Delta$ is simply-laced.
By Theorem~\ref{thm:2v-e=h}, 
\[
     2\dim\g(1)-\#\eus E(1)+2=kh ,
\]
where $k=\#\Pi(1)$ is the number of simple roots that are not orthogonal to $\theta$.
Furthermore,
$\dim\g(1)+\dim\g(2)=\#\{\gamma\in\Delta^+\mid (\gamma,\theta)>0\}$, and the latter is equal 
to $2h-3$ \cite[Ch. VI, \S\,1, n.\,11, prop. 32]{bo4-6}. Hence $\dim\g(1)=2h-4$ and
$\#\eus E(1)=(4-k)h-6$.
For $\GR{D}{n}$ and $\GR{E}{n}$, $\theta$ is fundamental, i.e., $k=1$.
Hence $\#\eus E(1)=3h-6=\frac{3}{2}\dim\g(1)$.
In particular, if  $\g=\GR{E}{8}$, then $\g(0)'=\GR{E}{7}$ and $\g(1)=\sfr(\varpi_1)$.
Whence $\#\eus E(\GR{E}{7},\varpi_1)=(3/2){\cdot}56=84$.
\end{ex}  

Looking at Eq.~\eqref{eq:iz-teor}, one might suspect 
that each summand in the left hand side in non-negative. Actually, each summand appears
to be positive, and this property is, in a sense, characteristic for 
\textsf{wmf}-representations associated with
$\BZ$-gradings.

\begin{thm}  \label{thm:2e-v>0}
Let $\g=\bigoplus_{i\in\BZ}\g(i)$ be an arbitrary $\BZ$-grading.
Then $0< 2\dim\g(1)-\#\eus E(1)\le h(\g)$.
Moreover, if $m=\max\{j\mid \g(j)\ne 0 \}>1$, then $2\dim\g(1)-\#\eus E(1)< h$.
\end{thm}
\begin{proof}
As is explained above, it suffices to consider $1$-standard $\BZ$-gradings. 
If  $\Pi(1)=\{\ap_i\}$, then  $m=[\theta:\ap_i]$.

\vskip1ex\noindent
\textbullet \quad Suppose $\Delta$ is simply-laced.

a) \ For $[\theta:\ap_i]=1$,  the assertion follows from Theorem~\ref{thm:2v-e=h} with $k=1$.

b) \ For $[\theta:\ap_i]=2$,  it is sufficient to prove that 
$0< 2\dim\g(2)-\#\eus E(2) < h$.

\noindent
We know that $\g(2)$ is a simple $\g(0)$-module. 
Let  $\es(0)$ be the reductive subalgebra of $\g(0)$ that contains the 
(one-dimensional) centre of
$\g(0)$ and all simple ideals acting non-triviallly on $\g(2)$.
Then $\es:=\g(-2)\oplus\es(0)\oplus \g(2)$ is a simple subalgebra of $\g$
and this decomposition is a short grading of $\es$.
Hence $2\dim\g(2)-\#\eus E(2) =h(\es)>0$ (Theorem~\ref{thm:2v-e=h}).
It remains to notice that the highest root $\theta\in\Delta(2)\subset\Delta^+$ 
is the highest root
for $\es$ as well, but the height of $\theta$ in $\Delta(\es)$ is strictly less than that in 
$\Delta(\g)$,
i.e., $h(\es)< h=h(\g)$.

c) \ So far the argument was satisfactory. In fact, it completely covers the series $\GR{A}{n}$
and $\GR{D}{n}$.
But for $[\theta:\ap_i]\ge 3$, we have to resort
to  case-by-case considerations. There is one such case for $\GR{E}{6}$,
three for $\GR{E}{7}$, and six for $\GR{E}{8}$. 

The output of our computations for all 1-standard gradings of all simple algebras
is presented after the proof.
Here is a sample of required computations. For  $\g=\GR{E}{8}$ and $\ap_4\in\Pi$, we have $[\theta:\ap_4]=5$. The corresponding coloured 
Dynkin  diagram is: 

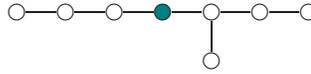
\begin{center}
\begin{picture}(135,40)(0,0)
\setlength{\unitlength}{0.017in}
\put(80,8){\line(0,1){9}}
\put(65,20){\color{my_color}\circle*{5}} 
\multiput(23,20)(15,0){6}{\line(1,0){9}}
\multiput(20,20)(15,0){7}{\circle{5}}
\put(80,5){\circle{5}}
\end{picture} 
\end{center}

\noindent
This diagram shows that, for the $\BZ$-grading determined by $\ap_4$,  
$\g(0)'$ is $\GR{A}{3}\times \GR{A}{4}$ and
the  highest weight of $\g(1)$ is $\varpi_1+\varpi_2'$. In other words, $\g(1)=
\sfr(\varpi_1)\otimes\sfr(\varpi_2')$ is the tensor
product of the tautological representation of $\GR{A}{3}$ and the second fundamental
representation of $\GR{A}{4}$. Hence $\dim\g(1)=40$.
Using Lemma~\ref{lm:tensor-wmf} and data in Table~\ref{tabl}, we compute that
$\#\eus E(1)=12{\cdot}4+10{\cdot}3=78$.

\vskip1ex\noindent
\textbullet \quad Suppose $\Delta$ is multiply-laced.
Here our argument is fully computational. 

\noindent
{\bf --}  \ The case of $\GR{C}{n}$. 
For $k\le n-1$, $[\theta:\ap_k]=2$  and the coloured Dynkin diagram is 

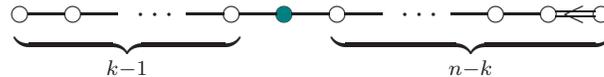
\begin{center}
\begin{picture}(300,22)(0,-2)
\multiput(30,8)(20,0){2}{\circle{6}}
\multiput(110,8)(20,0){3}{\circle{6}}
\multiput(210,8)(20,0){3}{\circle{6}}
\multiput(232.5,7)(0,2){2}{\line(1,0){15}}
\put(130,8){\color{my_color}\circle*{6}}
\multiput(93,8)(20,0){4}{\line(1,0){14}}
\multiput(33,8)(20,0){2}{\line(1,0){14}}
\multiput(193,8)(20,0){2}{\line(1,0){14}}
\put(74,5){$\cdots$}
\put(174,5){$\cdots$}
\put(235,5){$<$} 
\put(28,0){$\underbrace%
{\mbox{\hspace{85\unitlength}}}_{k-1}$}
\put(148,0){$\underbrace%
{\mbox{\hspace{105\unitlength}}}_{n-k}$}
\end{picture}
\end{center}

\vskip1.5ex\noindent 
Therefore, $\g(0)'$ is $\GR{A}{k-1}\times\GR{C}{n-k}$ and $\g(1)=\sfr(\varpi_1+\varpi_1')
=\sfr(\varpi_1)\otimes\sfr(\varpi_1')$.
Hence $\dim\g(1)=k(2n-2k)$ and $\#\eus E(1)=(k-1)(2n-2k)+k(2n-2k-1)
=2\dim\g(1)-(2n-k)$.

Finally, $[\theta:\ap_n]=1$ and this case is considered in Theorem~\ref{thm:2v1-e1=h}.

\noindent
{\bf --} \ Computations for $\GR{B}{n}$ are quite similar. For $k\ge 2$, $[\theta:\ap_k]=2$.
Then $\g(0)'$ is $\GR{A}{k-1}\times\GR{B}{n-k}$ and $\g(1)
=\sfr(\varpi_1)\otimes\sfr(\varpi_1')$.
Therefore, $\dim\g(1)=k(2n-2k+1)$ and $\#\eus E(1)=(k-1)(2n-2k+1)+k(2n-2k)
=2\dim\g(1)-(2n-k+1)$.

\noindent
{\bf --} \  Consider one possibility for $\GR{F}{4}$. Here $[\theta:\ap_3]=3$
and the coloured Dynkin diagram is 
\begin{center}
\begin{picture}(100,12)(-10,5)
\multiput(10,8)(20,0){4}{\circle{6}}
\put(50,8){\color{my_color}\circle*{6}}
\put(53,8){\line(1,0){14}}
\put(13,8){\line(1,0){14}}
\multiput(32.5,7)(0,2){2}{\line(1,0){15}}
\put(35,5){$<$} 
\end{picture}
\end{center}

\noindent 
Therefore, $\g(0)'= \GR{A}{2}\times\GR{A}{1}$ with $\g(1)=\sfr(2\varpi_1)\otimes\sfr(\varpi_1')$.
Hence $\dim\g(1)=12$ and $\#\eus E(1)=18$.
\end{proof}

\noindent
Hopefully,  there is a better proof of Theorem~\ref{thm:2e-v>0}.
The following table presents the numbers $2\dim\g(1)-\#\eus E(1)$ for all 1-standard
gradings of the exceptional Lie algebras.

\begin{center}
\begin{tabular}{c|cccccccc||c|cccc|}
   & $\ap_1$ & $\ap_2$ & $\ap_3$ & $\ap_4$ & $\ap_5$ & $\ap_6$ & $\ap_7$ & $\ap_8$ 
& &  $\ap_1$ & $\ap_2$ & $\ap_3$ & $\ap_4$ \\ \hline
$\GR{E}{6}$ & 12 & 6 & 3 &  6 & 12 & 10 &  -- & --  &
$\GR{F}{4}$ &   8 & 5 & 6 &  11\\    
$\GR{E}{7}$ & 18 & 8 & 4 &  2 &   5 & 16 & 10 & --  &
$\GR{G}{2}$ &   3 & 5 & -- &  -- \\    
$\GR{E}{8}$ & 28 & 9 & 4 &  2 &   1 &   3 & 16 &  7 & & & & & \\    \hline
\end{tabular}
\end{center}

\vskip1.5ex\noindent
Below, for the 1-standard $\BZ$-grading determined by the simple root $\ap_k$,  the
difference $2\dim\g(1)-\#\eus E(1)$ is denoted by $\mathcal  Z(\ap_k)$.  Here are data for the classical series:
\begin{itemize}
\item[$\GR{A}{n}$:] \   $\mathcal Z(\ap_k)=n+1=h$ for any $k$;
\item[$\GR{B}{n}$:] \   $\mathcal  Z(\ap_k)=2n-k+1$  for any $k$;
\item[$\GR{C}{n}$:] \   $\mathcal Z(\ap_k)=2n-k$ \ if $k\le n-1$; \quad  $\mathcal Z(\ap_n)=2n$.
\item[$\GR{D}{n}$:] \  $\mathcal Z(\ap_k)=2n-2k$ \ if $k\le n-2$; \quad  
                  $\mathcal Z(\ap_{n-1})=\mathcal Z(\ap_n)=2n-2$;
\end{itemize}

\noindent
It is curious that if $\ap_k$ is any vertex for $\GR{A}{n}$ or the branch vertex for $\GR{D}{n}$ or $\GR{E}{n}$, then 
$\mathcal Z(\ap_k)$ equals the determinant of the corresponding Cartan matrix.

\section{$\BZ_m$-gradings versus $\BZ$-gradings}   
\label{sec:periodic}

\noindent
We continue to assume that $\g$ is simple.
For a  $\BZ$-grading, we prove in Theorem~\ref{thm:2e-v>0} that 
$2\dim\g(1)> \#\eus E(1)$.
That is to say, the number of edges cannot be too large.
In the simply-laced case it turns out, rather surprisingly, that if we ``replace'' a 
$\BZ$-grading with a periodic grading, this inequality turns into equality!

More precisely, consider the following situation. Let $\vartheta\in\text{Aut\,}(\g)$ be an 
{\sl inner\/} automorphism of order $m> 1$. 
Choose a primitive root of unity $\zeta=\sqrt[m]{1}$ and set
$\g_i=\{x\in\g\mid \vartheta(x)=\zeta^ix\}$.
This yields a $\BZ_m$-grading of $\g$:
\[
    \g=\bigoplus_{i\in\BZ_m} \g_i .
\]
Here $\g_0$ is reductive. 
As $\vartheta$ is inner, $\rk\g=\rk\g_0$ and each $\g_i$ is a \textsf{wmf}
$\g_0$-module. Furthermore, we may assume that $\te\subset\g_0$
and then obtain the decomposition $\Delta=\sqcup_{i\in\BZ_m}
\Delta_i$, where $\Delta_i$ is the set of weights of the $\g_0$-module $\g_i$.
Let $\eus E_i$ denote the set of edges in the Hasse diagram of the
weight poset of $\g_i$. 

\begin{thm}   \label{thm:2v=e}
Let $\g=\bigoplus_{i\in\BZ_m} \g_i$ be a periodic grading such that $\g_1$ is a simple
$\g_0$-module. Suppose that $\Delta$ is simply-laced. Then
$2\dim\g_1=\#(\eus E_1)$.
\end{thm}\begin{proof}
We proceed in a case-by-case fashion. The periodic gradings such that $\g_1$ is 
a simple $\g_0$-module have a nice explicit description, see \cite[\S\,8]{vi76}, \cite[3.7]{t41}.
Any such grading is determined by one vertex of the {\it extended Dynkin diagram\/}
of $\g$, denoted $\tilde{\eus D}(\g)$ or $\GRt{X}{n}$ if $\GR{X}{n}$ is the Cartan type of $\g$.
The extra vertex of $\tilde{\eus D}(\g)$
is denoted by $\ap_0$ and we formally set $[\theta:\ap_0]=1$.
In figures below, the extra vertex is marked with 'cross'.
A choice of vertex  $\ap_i$, $0\le i\le n$, yields a $\BZ_m$-grading with
$m=[\theta:\ap_i]$. Having removed this vertex, one obtains a union of Dynkin 
diagrams that represents $\g_0$.
In this case $\g_0$ is semisimple.
In fact, $\g_0$ is semisimple if and only if $\g_1$ is a simple $\g_0$-module, see
\cite[Prop.~18]{vi76}. 
(For $[\theta:\ap_i]=1$, one obtains the initial Dynkin diagram and the trivial
automorphism. So, this case
is excluded from further considerations.)
The bonds between $\ap_i$ and the adjacent vertices of $\tilde{\eus D}(\g)$ 
determine the structure of the $\g_0$-module $\g_1$, see \cite[Prop.\,17]{vi76}
for the details.

D)  Consider the periodic grading of $\GR{D}{n}$ corresponding to $\ap_k$, $2\le k\le n-2$.
Here $[\theta:\ap_k]=2$ and the extended diagram is the following:

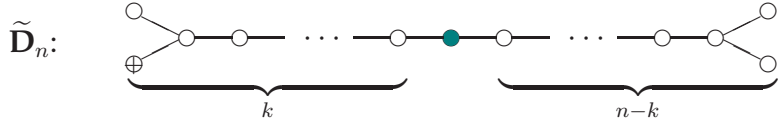
\raisebox{4ex}{ $\GRt{D}{n}$: }    \quad
\begin{picture}(300,47)(0,-20)
\multiput(30,8)(20,0){2}{\circle{6}}
\multiput(110,8)(20,0){3}{\circle{6}}
\multiput(210,8)(20,0){2}{\circle{6}}
\put(130,8){\color{my_color}\circle*{6}}
\multiput(93,8)(20,0){4}{\line(1,0){14}}
\multiput(33,8)(20,0){2}{\line(1,0){14}}
\multiput(193,8)(20,0){2}{\line(1,0){14}}
\put(74,5){$\cdots$}
\put(174,5){$\cdots$}
\put(13,0){\line(2,1){14}}
\put(13,16){\line(2,-1){14}}
\put(233,9){\line(2,1){14}}
\put(233,7){\line(2,-1){14}}
\multiput(10,-2)(0,20){2}{\circle{6}}
\put(6,-5){{\footnotesize $+$}}
\multiput(250,-2)(0,20){2}{\circle{6}}
\put(8,-7){$\underbrace%
{\mbox{\hspace{105\unitlength}}}_{k}$}
\put(148,-7){$\underbrace%
{\mbox{\hspace{105\unitlength}}}_{n-k}$}
\end{picture}

\vskip1ex\noindent 
Therefore $\g_0=\GR{D}{k}\times\GR{D}{n-k}$ and 
$\g_1=\sfr(\varpi_1)\otimes \sfr(\varpi'_1)$.
Hence $\dim\g_1=4k(n-k)$ and, by Lemma~\ref{lm:tensor-wmf}, $\#(\eus E_1)=8k(n-k)$.

E)  For the exceptional Lie algebras, E.B.\,Vinberg gives the table of all periodic gradings
such that $\g_0$ is semisimple \cite[\S\,9]{vi76}. It is not hard to check the required
equality for the inner automorphisms of $\GR{E}{n}$, $n=6,7,8$ (the first 14 items in Vinberg's
table).
As a sample, we consider the automorphism of order 5 for $\GR{E}{8}$ (cf. the case considered in the proof of Theorem~\ref{thm:2e-v>0}).
Now the extended coloured  Dynkin diagram is:
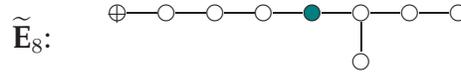
\begin{center}
\raisebox{2ex}{ $\GRt{E}{8}$: }    \quad
\begin{picture}(135,40)(0,0)
\setlength{\unitlength}{0.017in}
\put(80,8){\line(0,1){9}}
\put(1.7,17.5){{\footnotesize $+$}}
\put(65,20){\color{my_color}\circle*{5}} 
\multiput(8,20)(15,0){7}{\line(1,0){9}}
\multiput(5,20)(15,0){8}{\circle{5}}
\put(80,5){\circle{5}}
\end{picture}  
\end{center}

\noindent
Therefore $\g_0=\GR{A}{4}\times\GR{A}{4}$ and 
$\g_1=\sfr(\varpi_1)\otimes \sfr(\varpi'_2)$.
Hence $\dim\g_1=5{\cdot}10=50$ and $\#(\eus E_1)=12{\cdot}5+10{\cdot}4=100$.
\end{proof}

However, in the multiply-laced case, we still obtain the strict inequality 
$2\dim\g_1 > \#(\eus E_1)$, see below.
Probably, the relation $2\dim\g_1=\#(\eus E_1)$ could be adjusted somehow,
but I have no idea. Nevertheless, we have the following general assertion:

\begin{thm}    \label{thm: 2v>= e}
Let $\g=\bigoplus_{i\in\BZ_m} \g_i$ be an arbitrary periodic grading.
Then $2\dim\g_1\ge \#(\eus E_1)$ and the equality occurs if and only if 
the assumptions of Theorem~\ref{thm:2v=e} holds.
\end{thm}
\begin{proof}
By Vinberg's theory \cite[\S\,8]{vi76}, if we wish to check all possible representations
$(\g_0:\g_1)$
associated with periodic gradings, it suffices to deal with periodic gradings  determined by
arbitrary subsets $\Pi_1$ of the extended Dynkin diagram of $\g$. 
If $\Pi_1=\{\ap_{i_1},\dots,\ap_{i_k}\}$, $0\le i_j\le n$, then the order of $\vartheta$ equals
$\sum_j [\theta:\ap_{i_j}]$.
(Recall that  $\g_0$ is semisimple if and only if  $\#(\Pi_1)=1$.)

a)  \ $\#(\Pi_1)=1$ and $\Delta$ is simply-laced. This case is considered in 
Theorem~\ref{thm:2v=e}.

b)  \ $\#(\Pi_1)=1$ and $\Delta$ is multiply-laced. 

\noindent \textbullet \quad  
Consider the case of $\GR{C}{n}$ and 
$\Pi_1=\{\ap_k\}$, $1\le k \le n-1$. Since $[\theta:\ap_k]=2$, the automorphism $\vartheta$ is an involution.

 \qquad  $\GRt{C}{n}$: \quad
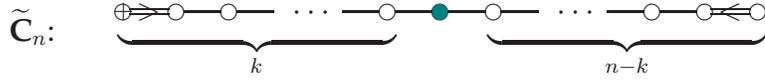
\begin{picture}(300,35)(0,-3)
\put(6.4,5.5){{\footnotesize $+$}}
\multiput(10,8)(20,0){3}{\circle{6}}
\multiput(110,8)(20,0){3}{\circle{6}}
\multiput(210,8)(20,0){3}{\circle{6}}
\multiput(12.5,7)(0,2){2}{\line(1,0){15}}
\multiput(232.5,7)(0,2){2}{\line(1,0){15}}
\put(130,8){\color{my_color}\circle*{6}}
\multiput(93,8)(20,0){4}{\line(1,0){14}}
\multiput(33,8)(20,0){2}{\line(1,0){14}}
\multiput(193,8)(20,0){2}{\line(1,0){14}}
\put(74,5){$\cdots$}
\put(174,5){$\cdots$}
\put(15,5){$>$} 
\put(235,5){$<$} 
\put(8,0){$\underbrace%
{\mbox{\hspace{105\unitlength}}}_{k}$}
\put(148,0){$\underbrace%
{\mbox{\hspace{105\unitlength}}}_{n-k}$}
\end{picture}

\vskip2ex\noindent
Here $\g_0=\GR{C}{k}\times\GR{C}{n-k}$ and $\g_1=\sfr(\varpi_1)\otimes\sfr(\varpi_1')$.
Therefore $\dim \g_1=2k(2n-2k)$ and $\#(\eus E_1)=(2k-1)(2n-2k)+2k(2n-2k-1)$.
Hence $2\dim\g_1-\#(\eus E_1)=2n$.

\noindent \textbullet \quad  
In the similar situation for $\GR{B}{n}$, with $\Pi_1=\{\ap_k\}$, $2\le k\le n$,
we obtain $\g_0=\GR{D}{k}\times\GR{B}{n-k}$ and $\g_1=\sfr(\varpi_1)\otimes\sfr(\varpi_1')$.
Then $\dim \g_1=2k(2n-2k+1)$ and $\#(\eus E_1)=2k(2n-2k+1)+2k(2n-2k)$.
Hence $2\dim\g_1-\#(\eus E_1)=2k$.

\noindent \textbullet  \quad  
For $\GR{F}{4}$, we consider the example with $\Pi_1=\{\ap_3\}$. The coloured extended diagram
is

\qquad  $\GRt{F}{4}$: \quad
\begin{picture}(100,25)(-10,5)
\multiput(10,8)(20,0){5}{\circle{6}}
\put(50,8){\color{my_color}\circle*{6}}
\multiput(53,8)(20,0){2}{\line(1,0){14}}
\put(13,8){\line(1,0){14}}
\multiput(32.5,7)(0,2){2}{\line(1,0){15}}
\put(35,5){$<$} 
\put(86.4,5.5){{\footnotesize $+$}}
\end{picture}
\vskip1.5ex\noindent
and $\vartheta$ is of order 3.
Here $\g_0=\GR{A}{2}\times\GR{A}{2}$ and $\g_1=\sfr(2\varpi_1)\otimes\sfr(\varpi_1')$.
Therefore $\dim \g_1=18$ and $\#(\eus E_1)=30$.
We omit consideration of other vertices for $\GR{F}{4}$ and the $\GR{G}{2}$-case.

c)  \ $\#(\Pi_1)=k\ge 2$.  

\noindent Recall that $\tilde{\eus D}(\g)$ is the extended Dynkin diagram of $\g$.
Suppose that $k$ vertices of $\tilde{\eus D}(\g)$ are coloured.
Removing these vertices yields a union of Dynkin diagrams, and this is 
the semisimple
part of $\g_0$. The centre of $\g_0$ is $(k-1)$-dimensional. Vinberg's theory says that
each vertex in $\Pi_1$ gives rise to an irreducible constituent of $\g_1$.
Namely, for $\ap_i\in \Pi_1$, let us take the largest connected subdiagram
of $\tilde{\eus D}(\g)$ that contains $\ap_i$ and no other vertices from $\Pi_1$.
(Practically, this means that we consider only those simple components
of $\g_0$ that act nontrivially on the selected simple submodule of $\g_1$.) 
Since $\#(\Pi_1)\ge 2$, we obtain a {\sl proper\/} subdiagram
of $\tilde{\eus D}(\g)$, which is thereby a Dynkin diagram.
That is, we get a {\sl usual\/} Dynkin (sub)diagram with one coloured vertex $\ap_i$.
As is explained in Section~\ref{sect:wmf-z}, this gives rise to a standard $\BZ$-grading
of a certain simple subalgebra $\es$ of $\g$. 
If $\es=\bigoplus_{i\in\BZ}\es(i)$ is the grading determined by $\ap_i$,
then our selected simple submodule of $\g_1$ is isomorphic
to the $\es(0)$-module $\es(1)$. However, we have already verified inequality 
$2\dim\es(1)> \#\eus E(1)$ for all standard $\BZ$-gradings.

Thus, for all periodic gradings with $\#(\Pi_1)\ge 2$, we have $2\dim\g_1> \#(\eus E_1)$.
\end{proof}

\begin{rmk}   \label{rmk:outer}
If $\vartheta\in \rm{Aut}(\g)$ is outer, then it may occasionally happen that $\g_1$ is a 
\textsf{wmf} $\g_0$-module. For instance, $\g=\GR{D}{4}$ has an automorphism of order 3 
such that $\g_0=\GR{G}{2}$ and $\g_1=\sfr(\varpi_1)$;
$\g=\GR{A}{2n+2m-1}$ has an automorphism of order 4 
such that $\g_0=\GR{D}{n}\times\GR{C}{m}$ and $\g_1=\sfr(\varpi_1)\otimes \sfr(\varpi_1')$.
\end{rmk}

\vskip1ex\noindent
The relations of Theorems~\ref{thm:2e-v>0}, \ref{thm:2v=e}, and \ref{thm: 2v>= e} 
suggest  us to determine all irreducible \textsf{wmf} representations
such that $\#\eus E(\lb)/\dim\sfr(\lb)\le 2$.
By Corollary~\ref{cor:additive}, the ratio 
\[
 \mathsf R:=\#\eus E(\lb)/\dim\sfr(\lb)
\] 
is additive with  respect to tensor
products. Therefore we can start with the irreducible representations of simple Lie algebras.
Using the last column of Table~\ref{tabl}, we can  find all suitable serial cases.

\begin{ex}    \label{ex:An-Dn}
{\bf --} \ For $\sfr(\GR{A}{n},\varpi_m)$, we can assume that $m\le (n+1)/2$.
Then 
\begin{itemize}
\item  \ $\mathsf R<2 $ if and only if $m=1,2$ and $n$ is arbitrary or $m=3$ and $n=5,6,7$;
\item  \ $\mathsf R=2$ if and only if  $(n,m)=(7,4)$ or $(8,3)$.
\end{itemize}

\noindent
{\bf --} \ For $\sfr(\GR{D}{n},\varpi_n)$, we see that $\mathsf R<2 $ if $n\le 7$
and $\mathsf R=2$ if $n=8$.

\noindent
{\bf --} \ For $\sfr(\GR{D}{n},\varpi_1)$, we always have $\mathsf R=1$.
\\
For these three series of representations,  $\sfr(\lb)$ occurs as $\g(1)$ 
for some 1-standard $\BZ$-grading (resp. as $\g_1$ for some periodic grading)
if and only if $\mathsf R< 2$ (resp. $\mathsf R= 2$).

\begin{center}
\begin{tabular}{lrc||lrc}
Repr. & $\mathsf R$ & grading & Repr. & $\mathsf R$ & grading  \\  \hline
$\sfr(\GR{A}{n}, \varpi_1)$ &  $<2$  &  $(\GR{A}{n+1}, \ap_1)$ &
$\sfr(\GR{A}{8}, \varpi_3)$ &  $2$  &  $(\GRt{E}{8}, \ap_8)$ \rule{0pt}{2.5ex} \\
$\sfr(\GR{A}{n}, \varpi_2)$ &  $<2$  &  $(\GR{D}{n+1}, \ap_{n+1})$ &
$\sfr(\GR{D}{n}, \varpi_n)$, $n=5,6,7$ &  $<2$  &  $(\GR{E}{n+1}, \ap_n)$ \\
$\sfr(\GR{A}{n}, \varpi_3)$, $n=5,6,7$ &  $<2$  &  $(\GR{E}{n+1}, \ap_{n+1})$ &
$\sfr(\GR{D}{8}, \varpi_8)$ &  $2$  &  $(\GRt{E}{8}, \ap_7)$ \rule{0pt}{2.5ex}\\
$\sfr(\GR{A}{7}, \varpi_4)$ &  $2$  &  $(\GRt{E}{7}, \ap_7)$ &
$\sfr(\GR{D}{n}, \varpi_1)$ &  $1$  &  $(\GR{D}{n+1}, \ap_1)$  \\ \hline
\end{tabular}
\end{center}
In column "grading", we point out the type of Dynkin diagram (usual or extended) and
the coloured vertex.
\end{ex}

\begin{ex}    \label{ex:n1-n2-n3}
Consider tensor products of simplest representations of algebras $\GR{A}{n}$.

\vskip.7ex
{\bf --} \ If  $\es=\GR{A}{n_1-1}\times\GR{A}{n_2-1}$ and
$\sfr(\varpi_1)\otimes\sfr(\varpi_1')$, then 
$\mathsf R=\frac{n_1-1}{n_1}+\frac{n_2-1}{n_2} < 2$. It is also seen that
this representation is associated with a short $\BZ$-grading of $\g=\GR{A}{n_1+n_2-1}$.

\vskip.7ex
{\bf --} \ If  $\es=\GR{A}{n_1-1}\times\GR{A}{n_2-1}\times\GR{A}{n_3-1}$ and
$\sfr(\varpi_1)\otimes\sfr(\varpi_1')\otimes\sfr(\varpi_1'')$, then 
$\mathsf R=\frac{n_1-1}{n_1}+\frac{n_2-1}{n_2}+\frac{n_3-1}{n_3}$.
In this case, condition $\mathsf R < 2$ can be rewritten as
\begin{equation}   \label{eq:n1-n2-n3}
          \frac{1}{n_1}+\frac{1}{n_2}+\frac{1}{n_3}> 1 .
\end{equation}  
This inequality often appears in classification problems; e.g., in classifying finite subgroups
of $SL_2(\BC)$ or quivers of finite type. 
The well-known solutions are: $(2,2,n)$,\,$(2,3,3)$,\, $(2,3,4)$,\,$(2,3,5)$.
The corresponding representations are associated with 
a $\BZ$-grading of $\GR{D}{n+2}$ or $\GR{E}{n}$, $n=6,7,8$ (the branch vertex of the 
Dynkin diagram should be coloured).

The solutions of equation $\mathsf R = 2$ are 
$(3,3,3), (2,4,4), (2,3,6)$. The corresponding representations are associated with 
a periodic grading of $\GR{E}{n}$ (the branch vertex of the extended Dynkin diagram
$\GRt{E}{n}$ should be coloured).

\vskip.7ex
{\bf --} \ For the product of four representations, the only  possibility is
$n_1=n_2=n_3=n_4=2$ with $\mathsf R = 2$. The corresponding representation 
is associated with the branch vertex of $\GRt{D}{4}$, i.e., with a $\BZ_2$-grading of
$\GR{D}{4}$.
Thus, we again observe the following phenomenon:
\\
{\it If $\mathsf R < 2$ (resp. $\mathsf R = 2$), then the \textsf{wmf}-representations 
in question are associated with $\BZ$-gradings (resp. periodic gradings).}
\end{ex}

It is tempting to suggest that this is always true. However,
the poset isomorphisms in Theorem~\ref{thm:isom-posets} limit one's optimism.
Indeed, $\eus P(\GR{A}{n}, \varpi_m)\simeq \eus P(\GR{A}{n-m+1}, m\varpi_1)$.
For $m=2$, the representations
$\sfr(\GR{A}{n}, \varpi_2)$ and $\sfr(\GR{A}{n-1}, 2\varpi_1)$ have isomorphic weight
posets and both are associated with short $\BZ$-gradings
(for $\g=\GR{D}{n+1}$ and $\GR{C}{n}$, respectively).
But for $m=3$ and $n=5,6,7$, the representation 
$\sfr(\GR{A}{n-2}, 3\varpi_1)$ is not associated with a grading,
whereas $\sfr(\GR{A}{n}, \varpi_3)$ is associated with a $\BZ$-grading of $\GR{E}{n+1}$.
Similar phenomenon occurs for the isomorphism
$\eus P(\GR{D}{n}, \varpi_n)\simeq \eus P(\GR{B}{n-1}, \varpi_{n-1})$ with $n=6,7,8$.

Therefore, a correct statement should be given in terms of weight posets.

\begin{thm}   \label{thm:posets-2v>=p}
Let $\eus P$ be the weight poset of an irreducible \textsf{wmf} representation
of a semisimple Lie algebra $\es$. Suppose every simple ideal of $\es$ is simply-laced.
Then $\eus P$ occurs as either $\Delta(1)$ for some $\BZ$-grading or
$\Delta_1$ for some periodic grading of a simple Lie algebra $\g$ if 
and only if \ $\mathsf R=\frac{\#\eus E(\eus P)}{\#\eus P}\le 2$.
Furthermore, $\frac{\#\eus E(\eus P)}{\#\eus P}=2$ if and only if $\eus P=\Delta_1$
for a periodic grading of a simply-laced algebra 
$\g$.
\end{thm}
\begin{proof}[Sketch of the proof]
In view of Theorems~\ref{thm:2e-v>0} and \ref{thm: 2v>= e}, 
only ``if'' part should be verified, i.e., if $\eus P$ is the weight poset of an irreducible 
\textsf{wmf}-representation such that $\frac{\#\eus E(\eus P)}{\#\eus P}\le 2$, then 
$\eus P\simeq \Delta(1)$ or $\eus P\simeq \Delta_1$ for some grading of $\g$.

All connected posets arising as $\Delta(1)$ or $\Delta_1$ are readily determined via 
Vinberg's theory:
one should remove one vertex from either  the usual or
extended Dynkin diagram of $\g$ and consider  the representations obtained,
modulo to equivalence described at the end of Section~\ref{sect:numb-edg}.
This yields the {\sl first list\/} of posets.

On the other hand, we begin with posets pointed out in 
Eq.~\eqref{align:posets}, except $\eus P(\GR{C}{3},\varpi_3)$.
The serial cases satisfying condition $\mathsf R\le 2$ are determined in 
Example~\ref{ex:An-Dn}. All these ``initial'' posets are contained in the first list.
Then, using Corollory~\ref{cor:additive}, we determine the cartesian products of them
that still satisfy condition $\mathsf R\le 2$.
Example~\ref{ex:n1-n2-n3} can be regarded as part of 
relevant argument. We omit other routine considerations.
Finally,  we will see that all admissible cartesian products of the initial posets belong to
the first list.
%
\end{proof}

\begin{rmk}
If we omit the condition that all simple factors are simply-laced, then the assertion of 
Theorem~\ref{thm:posets-2v>=p} becomes wrong.
Consider the \textsf{wmf}-representation $\sfr(\varpi_3)\otimes \sfr(\varpi_1')$ 
of $\es=\GR{C}{3}\times\GR{A}{2}$.  Here $\mathsf R=17/14 + 2/3 < 2$, but 
$\eus P(\varpi_3+\varpi_1')$ is not associated with a grading.
\end{rmk}

\begin{rmk}   \label{rmk:det_Cartan}
The appearance of inequality~\eqref{eq:n1-n2-n3} as an equivalent of the condition
$\#\eus E(1)/\dim\g(1) < 2$ for some  $\BZ$-gradings
suggests that at least in the simply-laced case 
there could be a direct relationship between the determinant of the Cartan matrix 
and the number $2\dim\g(1)-\#\eus E(1)$. This is also confirmed by the following
coincidence.
The periodic gradings, where $\g_1$ is a simple $\g_0$-module, satisfy the condition
$2\dim\g_1-\#(\eus E_1)=0$; on the other hand, periodic grading are described via extended
Dynkin diagrams, and the extended Cartan matrix has zero determinant.
\end{rmk}

\section{$\BZ$-gradings and upper covering polynomials}   
\label{sec:cover}

\noindent
For $\BZ$-gradings, 
yet another property of the $\g(0)$-modules $\g(i)$ can be expressed in terms of 
upper covering polynomials (see Section~\ref{sec:general})
of posets $\Delta(i)$.

\begin{prop}   \label{prop:deg3}
For any  $\BZ$-grading of $\g$, $\deg \eus K_{\Delta(i)}(t)\le 3$.
\end{prop}\begin{proof}
The posets $\Delta(i)$ are obtained from  $\Delta^+$ by removing certain edges. 
Clearly, this procedure does not increase the degree of upper covering polynomial,
and we use Theorem~\ref{thm:deg-coveri}.
\end{proof}

\begin{ex}   \label{ex:varpi4}
Consider the $\GR{A}{n}$-module $\sfr(\varpi_4)$ for $n\ge 7$.
Here the weight $\esi_1+\esi_3+\esi_5+\esi_7=:(1357)$ covers four weights:
$(2357), (1457), (1367), (1258)$. It follows that $\deg \eus K_{\eus P(\varpi_4)}\ge 4$.
Therefore these modules cannot occur in 
connection with $\BZ$-gradings. (Another reason is that 
$\#\eus E(\varpi_4)/\dim\sfr(\varpi_4)\ge 2$.)
\end{ex}

We have computed the upper covering polynomials for all the weight posets 
associated with Table~\ref{tabl}. 
\begin{thm}
Discarding the repetition of posets, the upper covering polynomials are:

\begin{itemize}
\item[\protect{$\GR{A}{n}$}:]  \quad   $\eus K_{\eus P(\varpi_m)}(t)=\sum_{r\ge 0} \genfrac{(}{)}{0pt}{}{m}{r}\genfrac{(}{)}{0pt}{}{n-m+1}{r}t^r$;
\item[\protect{$\GR{D}{n}$}:]  \quad   $\eus K_{\eus P(\varpi_n)}(t)=\sum_{r\ge 0} \genfrac{(}{)}{0pt}{}{n}{2r}t^r$; \rule{0pt}{3.5ex} \quad $\eus K_{\eus P(\varpi_1)}(t)=1+(2n-2)t+t^2$;
\item[\protect{$\GR{C}{3}$}:]  \quad  $\eus K_{\eus P(\varpi_3)}(t)=1+9t+4t^2$;
\item[\protect{$\GR{E}{6}$}:]  \quad  $\eus K_{\eus P(\varpi_1)}(t)=1+16t+10t^2$;
\item[\protect{$\GR{E}{7}$}:]  \quad  $\eus K_{\eus P(\varpi_1)}(t)=1+27t+27t^2+t^3$.
\end{itemize}
\end{thm}
\begin{proof}
\noindent The only non-trivial cases are the first two. We have bijective proofs that use the description of weights given in the proof of Theorem~\ref{thm:isom-posets}.

1) \  $\eus P(\GR{A}{n}, \varpi_m)$.
As in Section~\ref{sect:numb-edg}, let $\boldsymbol{i}=(i_1,\dots,i_m)$ be a weight.
Here $1\le i_1<\ldots < i_m\le n+1$.
Our goal is to realise when it is possible to subtract exactly $r$ simple roots from 
$\boldsymbol{i}$.  A string (of length $a$) is
a subsequence of $\boldsymbol{i}$ of the form $(i,i+1,\dots,i+a-1)$, $a\ge 1$, where
$i-1,i+a$ are not in $\boldsymbol{i}$.
Regard $\boldsymbol{i}$ as a the disjoint union of strings, separated by gaps.
The string is said to be {\it proper\/} if it does not contain $n+1$.
Clearly, each proper string provides a possibility to subtract a simple root, and vice versa.
Therefore, $\boldsymbol{i}$ covers exactly $r$ weights if and only if it contains $r$ 
proper strings, and perhaps one non-proper string.
Make the transform
\[
    \boldsymbol{i}=(i_1,\dots,i_m)\mapsto (i_1,i_2-1,\ldots, i_m-m+1)=: \tilde{\boldsymbol{i}} .
\]
Under this transform each string in $\boldsymbol{i}$ 
squeezes into one element of $[n-m+2]$, and different strings
squeeze into different elements. The non-proper
string, if it occurs, is squeezed into $\{n-m+2\}$. Using the usual notation for repetitions, 
we see that  $\boldsymbol{i}$ covers exactly $r$ weights if and only if 
the resulting multiset is of the form
\[
    \tilde{\boldsymbol{i}}=(j_1^{a_1},\dots,j_r^{a_r}, (n-m+2)^{a_{r+1}}),
\]
where $1\le j_1 < j_2< \ldots < j_r\le n-m+1$, \ 
$\sum_i a_i=m$, \ $a_i\ge 1$ if $i\le r$, \ and $a_{r+1}\ge 0$. 
Here the $a_i$'s are the lengths of the strings in $\boldsymbol{i}$.
Such a multiset is fully determined by two subsets
$\{j_1,\dots,j_r\}\subset [n-m+1]$ and $\{a_1,a_1+a_2,\dots,a_1+\ldots +a_r\}\subset [m]$.
Thus, the number of possibilities for such multisets $ \tilde{\boldsymbol{i}}$ equals
$\genfrac{(}{)}{0pt}{}{m}{r}\genfrac{(}{)}{0pt}{}{n-m+1}{r}$.

2) \  $\eus P(\GR{D}{n}, \varpi_n)$. 
Using the isomorphism of Theorem~\ref{thm:isom-posets}(2), we have to prove that,
for $\GR{B}{n}$, 
$\eus K_{\eus P(\varpi_n)}(t)=\sum_{r\ge 0} \genfrac{(}{)}{0pt}{}{n+1}{2r}t^r$. 
Recall that $\Pi(\GR{B}{n})=\{\esi_1-\esi_2,\dots,\esi_{n-1}-\esi_n,\esi_n\}$.
A weight $\mu=\frac{1}{2}(\pm\esi_1\pm\esi_2\ldots\pm\esi_n)\in \eus P(\GR{B}{n}, \varpi_n)$
can be regarded as arbitrary sequence of $n$ signs `$+$' and `$-$'.

\noindent
For $i < n$, we have $\mu-\ap_i\in \eus P(\GR{B}{n}, \varpi_n)$ if and only if 
the $i$-th sign is `$+$' and the next one is `$-$'; and  
$\mu-\ap_n\in \eus P(\GR{B}{n}, \varpi_n)$ if and only if  the last sign is `$+$'.
It follows  that if $\mu$ covers exactly $r$ weights, then this can be achieved in two ways:

(a) \ $\mu$ has exactly $r$  changes of signs of the form `$+-$'  and the last sign is 
`$-$'.  

(b) \  $\mu$ has exactly $r-1$  changes of signs of the form `$+-$' and the last sign is 
`$+$'.

In case (a), $\mu=
(\underbrace{-\dots -}_{c}\underbrace{+\dots +}_{a_1}\underbrace{-\dots -}_{b_1}\cdots 
\underbrace{+\dots +}_{a_r}\underbrace{-\dots -}_{b_r})$, where 
$c\ge 0$, $a_i,b_i>0$ and $c+\sum(a_i+b_i)=n$. The starting positions of $a$- and $b$-strings
can be arbitrary. Hence there are $\genfrac{(}{)}{0pt}{}{n}{2r}$ possibilities for such $\mu$.

In case (b), 
$\mu=
(\underbrace{-\dots -}_{c}\underbrace{+\dots +}_{a_1}\underbrace{-\dots -}_{b_1}\cdots 
\underbrace{+\dots +}_{a_r})$, and here we have $\genfrac{(}{)}{0pt}{}{n}{2r-1}$ possibilities.
Altogether, we obtain $\genfrac{(}{)}{0pt}{}{n}{2r}+\genfrac{(}{)}{0pt}{}{n}{2r-1}=
\genfrac{(}{)}{0pt}{}{n+1}{2r}$, as required.
\end{proof}

To compute upper covering polynomials for irreducible \textsf{wmf}-representations of 
semisimple Lie algebras, one can use the following refinement of Lemma~\ref{lm:tensor-wmf}.

\begin{prop}
Let $(\g', \sfr')$, $(\g'',\sfr'')$ be two  \textsf{wmf\/} representations. Then
\[
   \eus K_{\eus P(\sfr'\otimes\sfr'')}(t)=\eus K_{\eus P(\sfr')}(t)\cdot \eus K_{\eus P(\sfr'')}(t) .
\]
\end{prop}
\begin{proof}
This readily follows from the definition of upper covering polynomials and the cartesian 
product of graphs.
\end{proof}

Properties of upper covering polynomials provide another possible approach to the proof of
inequality $2\dim\g(1)-\#\eus E(1)>0$ in Theorem~\ref{thm:2e-v>0}.
In view of Proposition~\ref{prop:deg3}, we can write
$\eus K_{\Delta(1)}(t)=a_0+a_1t+a_2t^2+a_3t^3$.  Here $a_i=\#\{\mu\in\Delta(1)\mid
\mu \text{ covers $i$ elements}\}$. 
Then
\[
  2\dim\g(1)-\#\eus E(1)=2\eus K_{\Delta(1)}(1)-\eus K_{\Delta(1)}'(1)=2+a_1-a_3.
\]
This integer is automatically positive if $a_3=0$, i.e., $\deg\eus K_{\Delta(1)}\le 2$.
Since  the latter is the case for $\GR{A}{n},\GR{B}{n},\GR{C}{n},$ and $\GR{G}{2}$, 
we obtain a proof of the above inequality for these series. 
In general, it would be desirable to find an {\sl a priori\/} proof of the inequality 
$2+a_1-a_3>0$ for all posets $\Delta(1)$.

\end{document}